\documentclass[reqno,a4paper,12pt]{amsart} 

\usepackage{amsmath,amscd,amsfonts,amssymb}
\usepackage{mathrsfs,dsfont}

\numberwithin{equation}{section}
\numberwithin{figure}{section}

\newcommand\R{\mathbb{R}}

\newcommand\Z{\mathbb{Z}}

\newcommand\G{\mathbf{G}}

\newcommand\lam{\lambda}
\newcommand\Lam{\Lambda}

\newcommand\Om{\Omega}

\newcommand\1{\mathds{1}}

\renewcommand\leq{\leqslant}
\renewcommand\geq{\geqslant}
\newcommand\sbt{\subset}

\newcommand{\ft}[1]{\widehat #1}
\newcommand{\dotprod}[2]{\langle #1 , #2 \rangle}
\newcommand{\mes}{\operatorname{mes}}

\theoremstyle{plain}
\newtheorem{thm}{Theorem}[section]

\newtheorem{corollary}[thm]{Corollary}

\newtheorem{conj}[thm]{Conjecture}

\newtheorem*{claim*}{Claim}

\newcommand{\thmref}[1]{Theorem~\ref{#1}}
\newcommand{\secref}[1]{Section~\ref{#1}}

\newcommand{\conjref}[1]{Conjecture~\ref{#1}}

\theoremstyle{definition}

\newtheorem*{definition*}{Definition}
\newtheorem*{remarks*}{Remarks}
\newtheorem*{remark*}{Remark}

\newenvironment{enumerate-roman}
{\begin{enumerate}
\addtolength{\itemsep}{5pt}
}
{\end{enumerate}}

\newenvironment{enumerate-alph}
{\begin{enumerate}
\addtolength{\itemsep}{5pt}
}
{\end{enumerate}}

\newenvironment{enumerate-num}
{\begin{enumerate}
\addtolength{\itemsep}{5pt}
}
{\end{enumerate}}

\newenvironment{enumerate-text}
{\begin{enumerate}
\addtolength{\itemsep}{5pt}
}
{\end{enumerate}}

\begin{document}

\title{Gabor orthonormal bases, tiling and periodicity}

\author{Alberto Debernardi Pinos}
\address{Department of Mathematics, University of Aveiro, Aveiro 3810-193, Portugal}
\email{adebernardipinos@gmail.com}

\author{Nir Lev}
\address{Department of Mathematics, Bar-Ilan University, Ramat-Gan 5290002, Israel}
\email{levnir@math.biu.ac.il}

\date{December 2, 2021}
\thanks{Research supported by ISF Grants No.\ 447/16, 227/17 and 1044/21 and 
ERC Starting Grant No.\ 713927.
The first author was also partially supported by Ministry of Education and 
Science of the Republic of Kazakhstan (AP08053326), and by The Center for 
Research \& Development in Mathematics and Applications, through the 
Portuguese Foundation for Science and Technology (UIDP/04106/2020).}

\begin{abstract}
We show that if the Gabor system $\{ g(x-t) e^{2\pi i s x}\}$, $t \in T$, $s \in S$, is an orthonormal basis  in $L^2(\mathbb{R})$ and if the window function $g$ is compactly supported, then both the time shift set $T$ and the frequency shift set $S$ must be periodic. To prove this we establish a necessary functional tiling type condition for Gabor orthonormal bases which may be of independent interest.
\end{abstract}

\maketitle


\section{Introduction} \label{secI1}

\subsection{}
Let $g$ be a function in $L^2(\R^d)$, and let $T$ and $S$ be two countable 
sets in $\R^d$. The system of functions
\begin{equation}
\label{eqA1.1}
	\G(g,T,S):=\{ g(x-t) e^{2\pi i\dotprod{s}{x}} : t \in T, s \in S\}
\end{equation}
is called the \emph{Gabor system generated by the window
function $g$, the time shift set $T$ and the frequency shift set $S$}.
Gabor systems have been extensively studied in the context of
signal processing and time-frequency analysis, see e.g. \cite{Chr16}.

There is a long-standing problem
concerning the structure 
of the Gabor systems \eqref{eqA1.1}
 that constitute an orthonormal basis in the space
$L^2(\R^d)$. It is generally expected that
orthonormal bases of the form \eqref{eqA1.1}
must have a certain rigid structure.
This problem has been studied by different
authors, see e.g.\  \cite{CCJ99}, \cite{Liu01}, \cite{LW03},
\cite{DL14}.

In particular, it was proved in \cite[Theorem 1.1]{LW03} that if 
$\G(g, T, S)$ is an ortho\-normal basis in $L^2(\R)$
 where the function $g$ has compact support,
and if the frequency shift set $S$ is periodic,
then the time shift set $T$ must be periodic as well.
In the present paper we improve this
 result by establishing that the periodicity
of both 
the time and the
 frequency shift sets 
$T$ and $S$ follows without  any a priori 
assumption on the structure of these sets. 
We  will prove the following:

\begin{thm}
\label{thmA3}
Let $\G(g, T, S)$ be an orthonormal basis in $L^2(\R)$,
where $g \in L^2(\R)$ has compact support.
Then both $T$ and $S$ must be periodic sets, i.e.\ they have
the form 
\begin{equation}
\label{eqA3.2}
T = a \Z + \{t_1, \dots, t_n\}, \quad
S = b \Z + \{s_1, \dots, s_m\}.
\end{equation}
\end{thm}

The result is not true in dimensions two and higher,
see e.g.\ \cite[Example 3.1]{LW03}.

\subsection{}
Recall that a set $\Lam \sbt \R^d$ is said to be
\emph{uniformly distributed} if there exists
a number $D(\Lam)$ satisfying
 \begin{equation}
\label{eqA5.1}
\frac{\#(\Lambda\cap (x + [0,r]^d))}{r^d} =  D(\Lam)  + o(1), \quad r \to +\infty
 \end{equation}
uniformly with respect to $x \in \R^d$.
In this case, $D(\Lam)$ is called the
\emph{uniform density} of   $\Lam$.

It is known, see e.g.\ \cite[Lemma 2.2]{LW03}, that
 if $\G(g, T, S)$ is an orthonormal basis in $L^2(\R^d)$
then both  $T$ and $S$
must  be
uniformly distributed sets, and moreover
their uniform densities are positive and finite
numbers satisfying the relation
\begin{equation}
\label{eqD1.4.9}
D(T) D(S) = 1.
\end{equation}

Our approach to the proof of \thmref{thmA3}
is based on the  following  result,
which will also be proved in this paper
and which may be of independent interest:

\begin{thm}
\label{thmD1.4.3}
For any orthonormal basis $\G(g, T, S)$ in $L^2(\R^d)$ we have
\begin{equation}
\label{eqD1.4.1}
\sum_{t \in T} |g(x-t)|^2 = D(T) \quad \text{a.e.}
\end{equation}
and
\begin{equation}
\label{eqD1.4.2}
\sum_{s \in S} |\ft{g}(\xi-s)|^2 = D(S) \quad \text{a.e.}
\end{equation}
\end{thm}

Here $\ft{g}$ is 
 the Fourier transform of $g$
defined by   $\ft{g}(\xi) = \int g(x) \exp(-2 \pi i \dotprod{\xi}{x}) dx$.

The conclusions
\eqref{eqD1.4.1} and
\eqref{eqD1.4.2} 
can be stated in terms of
  translational tilings,
i.e.\ the function
$|g|^2$ tiles $\R^d$ at a constant level $D(T)$
by translations 
 with respect to the set $T$,
 while  $|\ft{g}|^2$  tiles at level $D(S)$ with respect to   $S$.
For an account on the theory of tiling by translates of a function
we refer the reader to \cite{KL21} and the references therein.

In dimension $d=1$, the conclusion 
\eqref{eqD1.4.1}
was obtained in \cite[Theorem 3.2]{CCJ99}
in the special case where both  $T$ and $S$ are 
arithmetic progressions. 
In  \cite[Theorem 4.4]{Cas04} this result
was extended to the case where 
$S$ is an arithmetic progression, but
$T$ is an arbitrary set in $\R$. \thmref{thmD1.4.3}
establishes that both
\eqref{eqD1.4.1} and
\eqref{eqD1.4.2} are valid in all
dimensions $d$ and without imposing any a priori
 assumptions on the structure of the time or the frequency shift sets 
$T$ and $S$.

As an immediate consequence of \thmref{thmD1.4.3} we obtain:

\begin{corollary}
\label{corD1.7}
If $\G(g, T, S)$ is an orthonormal basis  in $L^2(\R^d)$
then both $g$ and $\ft{g}$ must belong to $L^\infty(\R^d)$,
 and moreover we have
\begin{equation}
\label{eqD1.3}
\|g\|_{L^\infty(\R^d)} \leq \sqrt{D(T)}, \qquad  
\|\ft{g}\|_{L^\infty(\R^d)} \leq \sqrt{D(S)}.
\end{equation}
\end{corollary}

The proofs of Theorems
\ref{thmA3}  and  \ref{thmD1.4.3}
are given in Sections   \ref{sec:GONBT} and
 \ref{sec:GONBP} below.
In the last \secref{secR1} we give,
as another application of \thmref{thmD1.4.3}, a new
proof of a result obtained in \cite{DL14} on 
 the structure of the Gabor ortho\-normal bases
$\G(g, T, S)$ 
with a nonnegative window function $g$.

\maketitle


\section{Gabor orthonormal bases and tiling}
\label{sec:GONBT}

In this section we prove \thmref{thmD1.4.3}.
We assume that  $g \in L^2(\R^d)$,  that
$T, S$ are two countable 
sets in $\R^d$, and that the system $\G(g,T,S)$
is an orthonormal basis in $L^2(\R^d)$.
 We will show that conditions
\eqref{eqD1.4.1} and \eqref{eqD1.4.2} must be satisfied.

\begin{proof}[Proof of  \thmref{thmD1.4.3}]
Since the system \eqref{eqA1.1}
is an orthonormal basis in $L^2(\R^d)$
then, by Parseval's equality, for any $f \in L^2(\R^d)$ we   have
\begin{equation}
\label{eqD2.1}
\sum_{s \in S} \sum_{t \in T} \Big|
\int_{\R^d} f(x) \, \overline{g(x-t)} \, e^{-2 \pi i \dotprod{s}{x}} dx
\Big|^2 = \|f\|^2.
\end{equation}
Fix the function $f$ and for each $t \in T$ define the function
\begin{equation}
\label{eqD2.1.5}
h_t(x) := f(x) g(x-t)
\end{equation}
which is in $L^1(\R^d)$.
By applying \eqref{eqD2.1} to the function 
$\overline{f(x)} \, e^{2 \pi i \dotprod{\xi}{x}}$ in place of
$f(x)$, we obtain
\begin{equation}
\label{eqD2.2}
\sum_{s \in S} \sum_{t \in T} 
| \ft{h}_t (\xi-s) |^2 = \|f\|^2, \quad
\xi \in \R^d.
\end{equation}
In what follows we will assume that $f$ is not
the zero function.  Define
\begin{equation}
\label{eqD2.3}
H(\xi) :=  \|f\|^{-2} \sum_{t \in T} 
| \ft{h}_t (\xi) |^2, 
\end{equation}
then the condition \eqref{eqD2.2} can be
stated as 
\begin{equation}
\label{eqD2.3.5}
\sum_{s \in S} H(\xi-s) = 1, \quad \xi \in \R^d.
\end{equation}
In other words, the function $H$ tiles $\R^d$
at a constant level $1$ by translations with respect to the set 
$S$. We will now use the following known result:

\begin{thm}[{\cite{KL96}}]
\label{thmE1.1}
Let $H$ be a nonnegative measurable function on $\R^d$,
 and let $S$ be a countable set in $\R^d$.
If the tiling condition \eqref{eqD2.3.5} holds
then $S$ has a finite uniform density satisfying $D(S) = (\int H)^{-1}$.
\end{thm}

(Strictly speaking, 
this was proved in 
\cite[Section 2]{KL96} in dimension $d=1$,
under the extra assumption
that $H \in L^1(\R)$,
and for a weaker notion of 
density (i.e.\ not the uniform density).
The full statement of \thmref{thmE1.1}
can be  established by appropriate adjustments to
the proof in \cite{KL96}.)

In the present case the frequency shift set $S$ is known to have
positive uniform density, that is, $D(S) > 0$.
Hence \thmref{thmE1.1} implies
that the function $H$ must be in $L^1(\R^d)$. Moreover, 
using \eqref{eqD1.4.9} we conclude that
 $\int H =  D(T)$.

  On the other hand,
using
\eqref{eqD2.1.5}, \eqref{eqD2.3} and Plancherel's equality we obtain
\begin{align}
\|f\|^{2} \int_{\R^d}  H(\xi) d\xi &=
 \sum_{t \in T} \int_{\R^d} | \ft{h}_t (\xi) |^2 d\xi
= \sum_{t \in T}  \int_{\R^d} | h_t(x) |^2 dx\\
&= \int_{\R^d} |f(x)|^2 \sum_{t \in T}   | g(x-t) |^2 dx.
\end{align}
We thus arrive at the equality
\begin{equation}
\int_{\R^d} |f(x)|^2 \sum_{t \in T}   | g(x-t) |^2 \, dx
= D(T) \|f\|^{2}
\end{equation}
which holds for an arbitrary nonzero $f \in L^2(\R^d)$.
This implies that 
\eqref{eqD1.4.1} must be true.

Next we use the fact that the Fourier transform maps
an orthonormal basis in $L^2(\R^d)$ onto
another orthonormal basis,  from which it follows that
the system 
 $\G(\ft{g}, S, T)$ is also an ortho\-normal basis in $L^2(\R^d)$.
So we may apply the same considerations above   to the system
$\G(\ft{g},S,T)$ in place of $\G(g,T,S)$, which
implies that  \eqref{eqD1.4.2} must also be true
and completes
the proof of \thmref{thmD1.4.3}.
\end{proof}


\section{Gabor orthonormal bases and periodicity}
\label{sec:GONBP}

Next we prove \thmref{thmA3}, which states that
if the system $\G(g, T, S)$ is an ortho\-normal basis in $L^2(\R)$
and the function
$g$ has compact support, then 
both $T$ and $S$ must be periodic sets.
This result is true only in dimension $d=1$ and 
does not extend to higher dimensions,
see e.g.\ \cite[Example 3.1]{LW03}.

\begin{proof}[Proof of \thmref{thmA3}]
First we establish the  periodicity of the time shift 
set $T$. Define the function $\varphi(x) := D(T)^{-1} |g(x)|^2$
which is in $L^1(\R)$.
Due to \thmref{thmD1.4.3} we know that
condition \eqref{eqD1.4.1} must hold, i.e.\ $\varphi$
 tiles  $\R$ at a constant level $1$
by translations with respect to the set $T$.
We now  invoke another result from \cite{KL96}:

\begin{thm}[{\cite{KL96}}]
\label{thmE2.1}
Let $\varphi$ be  a nonnegative, compactly supported
 function in $L^1(\R)$. If $\varphi$ tiles at level one by translations
with respect to some set $T \sbt \R$
then $T$ must be a disjoint
union of  finitely many arithmetic progressions, 
namely
\begin{equation}
\label{eqE2.2}
T = \biguplus_{j=1}^{N}
(a_j \Z + b_j)
\end{equation}
where $a_j, b_j$ are real numbers and
$a_j>0$.
\end{thm}

In the present case the 
function $g$ has compact support, and hence
the same is true for $\varphi$ and
\thmref{thmE2.1} applies. We conclude
that the time shift set $T$ is of the form \eqref{eqE2.2}.
On the other hand it is well-known that $T$ must be
a \emph{uniformly discrete} set, i.e.\ there is $\delta>0$
such that $|t' - t| \geq \delta$ for any two distinct
points $t, t' \in T$. (This follows from the
orthogonality of the two functions $g(x-t)e^{2 \pi i sx}$
and $g(x-t')e^{2 \pi i sx}$, where $s$ is any
element of the frequency shift set $S$.)
Hence using Kronecker's theorem it follows 
that all the ratios $a_i / a_j$ must be
rational numbers. In turn this implies that $T$
is a periodic set with the structure
as in \eqref{eqA3.2} (see e.g.\ \cite[p.\ 670]{KL96}).

Next we establish the  periodicity of the frequency shift set $S$. 
For this purpose we will use the condition
\eqref{eqD1.4.2} from \thmref{thmD1.4.3},
which can be stated by saying that the function
 $\psi(\xi) := D(S)^{-1} |\ft{g}(\xi)|^2$
tiles $\R$ at level $1$ by translations
with respect to  $S$.  We observe however that 
\thmref{thmE2.1} does not apply to this tiling
since the Fourier transform $\ft{g}$ is not a
 compactly supported function, being  analytic 
 (this can be seen as a version
of the classical uncertainty principle).
We also note that there do exist non\-periodic
tilings of $\R$
by translates of a nonnegative   $L^1(\R)$ function 
of unbounded support, see \cite{KL16}.

 To address this issue we will use a different result obtained in 
 \cite{IK13, KL16}. To state the result we 
first recall that a set $S \sbt \R$ is said to have
\emph{finite local complexity} if $S$ can be 
enumerated as a sequence $\{s_n\},$ $n\in \Z$, such that $s_n<s_{n+1}$ and the successive differences $s_{n+1} - s_n$ take only finitely many different values.
The following result establishes that translational tilings
of finite local complexity
must be periodic, even if the function which tiles does not have
 compact support:

\begin{thm}[{\cite{IK13}, \cite{KL16}}]
\label{thmKL16FLC}
Let $S \sbt \R$ have finite local complexity.
If a function $\psi \in L^1(\R)$ tiles
at level one by translations with respect to $S$, then $S$
must be a periodic set, namely, it has the form $S = b\Z + \{s_1, \dots, s_m\}$.
\end{thm}

It would therefore suffice to show that in our case the set
$S$ has finite local complexity. 
It is easy to verify that  $S$ has  finite local complexity if and only if
$S$ has bounded gaps and the set of differences $S-S$ is
a discrete closed set in $\R$. Since we have  
$D(S)>0$ then $S$ indeed has bounded gaps.
For any fixed $t \in T$, the orthogonality of the system
$\{ g(x-t) e^{2\pi i sx} :  s \in S\}$
implies that the set $(S - S) \setminus \{0\}$ is contained in
the zero set of the Fourier transform of the $L^1(\R)$
function $|g|^2$. But since $|g|^2$ is a compactly
supported function, its Fourier transform is the restriction to
$\R$ of a nonzero entire function and hence the set of its real zeros
is a discrete closed set in $\R$. This establishes that $S$ has
finite local complexity, and thus the periodicity of $S$
follows from \thmref{thmKL16FLC}.
\end{proof}


\section{Application to a conjecture of Liu and Wang}
\label{secR1}

There is an interesting 
 conjecture due to Youming Liu and Yang Wang \cite{LW03}
which 
concerns the structure 
of the Gabor orthonormal bases $\G(g,T,S)$
such that the window
function $g$ has compact support.

To state the conjecture we first recall that
if $\Om \sbt \R^d$ is a set of positive and finite measure  
and if $\Lam \sbt \R^d$ is a countable set, then
$(\Om,\Lam)$ is called a \emph{tiling pair}
if the family of sets $\Omega+\lam$ $(\lam \in \Lam)$
forms a partition of $\R^d$ up to measure zero;
and that
$(\Om,\Lam)$ is called a \emph{spectral pair}
if the  system of exponential functions
\begin{equation}
\label{eqR1.3}
E(\Lam) := \{e^{2\pi i \dotprod{\lam}{x}} : \lam \in \Lam\}
\end{equation}
forms an orthogonal basis in the space $L^2(\Om)$.

\begin{conj}[{\cite{LW03}}]
\label{conj_lw}
Let $\G(g, T, S)$ be an orthonormal basis in $L^2(\R^d)$
such that the function $g$ has compact support. Then
the following three conditions must hold:
\begin{enumerate-num}
\item \label{lwc:i}
$|g(x)| = (\mes  \Om)^{-1/2} \, \1_\Om (x)$ 
a.e.\ for some measurable set $\Om \sbt \R^d$;
\item \label{lwc:ii}
$(\Om,T)$ is a tiling pair;
\item \label{lwc:iii}
$(\Om,S)$ is a spectral pair.
\end{enumerate-num}
\end{conj}

It is easy to verify, see e.g.\ \cite[Lemma 3.1]{LW03}, that 
any Gabor system   $\G(g, T, S)$ satisfying the
conditions  \ref{lwc:i}, \ref{lwc:ii} and \ref{lwc:iii} above 
forms an orthonormal basis
in $L^2(\R^d)$. The 
Liu-Wang conjecture asserts that 
if $g$ has compact support, then
these sufficient
conditions are also necessary for the system
 $\G(g, T, S)$ to be an orthonormal basis.

\conjref{conj_lw} is known to hold
in several special cases.  In \cite{Liu01}
the conjecture was proved in dimension one
and in the case where $T=S=\Z$. A similar result
is true also in $\R^d$. In
\cite[Theorem 1.2]{LW03}
the conjecture was proved in 
dimension one without
any assumption on the structure of the sets
$T$ and $S$ but assuming that
 the set $\Om = \{x : g(x) \neq 0\}$ 
 coincides a.e.\ with an interval in $\R$.
In \cite{DL14} 
the conjecture was proved in the special case
where $g \in L^2(\R^d)$ is a nonnegative function
(in this special case, the conjecture is true
even if $g$ is not  assumed to have
compact support).

\thmref{thmD1.4.3} allows us to give a new
proof of the latter result, namely:

\begin{thm}[{\cite{DL14}}]
\label{thmD7.1}
Let $g$ be a nonnegative function in $L^2(\R^d)$. If
$\G(g, T, S)$ is an orthonormal basis,
then the conclusions \ref{lwc:i}, \ref{lwc:ii} and \ref{lwc:iii} 
in \conjref{conj_lw} hold.
\end{thm}

\begin{proof}
Let $\Om = \{x : g(x) \neq 0\}$.
Since $g$ is a nonnegative function, the orthogonality of the system
$\{ g(x-t) e^{2\pi i \dotprod{s}{x}} :  t \in T\}$
where $s$ is any fixed 
 element of $S$, 
implies that the sets $\Omega+t$ $(t \in T)$
must be pairwise disjoint up to measure zero.
(This is the only place in the proof where the nonnegativity of $g$
is in fact used; the rest of the argument is valid for an arbitrary $g$.)
But the sets $\Omega+t$ are respectively the supports
of the  functions $|g(x-t)|^2$, $t \in T$, so
it follows from the tiling condition
\eqref{eqD1.4.1} that
$(\Om,T)$ is a tiling pair
and 
that $|g(x)|^2 = D(T)$
a.e.\ on $\Om$.
In turn this implies that
$D(T) \mes(\Om) = \int |g(x)|^2 dx = 1$
and hence  \ref{lwc:i} and \ref{lwc:ii} are established.

The fact that 
$(\Om,T)$ is a tiling pair
implies that for any fixed $t \in T$,
the system
$\{g(x-t) e^{2 \pi i \dotprod{s}{x}} : s \in S\}$
 is an orthonormal basis in the space $L^2(\Om+t)$.
But since the weight function
$g(x-t)$ has constant modulus a.e.\ on its support $\Om+t$, 
this implies that
also the
unweighted system of exponentials  $E(S)$ must 
 be an orthogonal  basis  in
 $L^2(\Om + t)$. This establishes \ref{lwc:iii}
and completes the proof.
\end{proof}

We conclude the paper by mentioning the references
\cite{GLW15}, \cite{CL18}, \cite{IM18}, \cite{LM19}, \cite{AAK20}
where other questions in the spirit of the Liu-Wang
conjecture are studied. Similar questions in 
the finite group setting are considered in
\cite{IKLMP21}.



\begin{thebibliography}{IKLMP21}

\bibitem[AAK20]{AAK20}
E. Agora, J. Antezana, M. N. Kolountzakis, 
Tiling functions and Gabor orthonormal basis. 
Appl. Comput. Harmon. Anal. \textbf{48} (2020), no. 1, 96--122.

\bibitem[Cas04]{Cas04}
P. G. Casazza, 
An introduction to irregular Weyl-Heisenberg frames. 
Sampling, wavelets, and tomography, 61--81, 
Appl. Numer. Harmon. Anal., Birkh\"{a}user, 2004.

\bibitem[CCJ99]{CCJ99}
P. G. Casazza, O. Christensen, A. J. E. M. Janssen, 
Classifying tight Weyl-Heisenberg frames. 
The functional and harmonic analysis of wavelets and frames (San Antonio, TX, 1999), 131--148, Contemp. Math., \textbf{247}, Amer. Math. Soc., Providence, RI, 1999.

\bibitem[Chr16]{Chr16}
O. Christensen,
An introduction to frames and Riesz bases,
second edition. 
Birkh\"{a}user, 2016.

\bibitem[CL18]{CL18}
R. Chung, C.-K. Lai, 
Non-symmetric convex polytopes and Gabor orthonormal bases. 
Proc. Amer. Math. Soc. \textbf{146} (2018), no. 12, 5147--5155. 

\bibitem[DL14]{DL14}
D. E. Dutkay, C. K. Lai, 
Uniformity of measures with Fourier frames. 
Adv. Math. \textbf{252} (2014), 684--707.

\bibitem[GLW15]{GLW15}
J.-P. Gabardo, C.-K. Lai, Y. Wang, 
Gabor orthonormal bases generated by the unit cubes. 
J. Funct. Anal. \textbf{269} (2015), no. 5, 1515--1538. 

\bibitem[IK13]{IK13}
A. Iosevich, M. N. Kolountzakis, 
Periodicity of the spectrum in dimension one. 
Anal. PDE \textbf{6} (2013), no. 4, 819--827.

\bibitem[IKLMP21]{IKLMP21}
A. Iosevich, M. N. Kolountzakis, Yu. Lyubarskii, A. Mayeli, J. Pakianathan,
On Gabor orthonormal bases over finite prime fields. 
Bull. Lond. Math. Soc. \textbf{53} (2021), no. 2, 380--391.

\bibitem[IM18]{IM18}
A. Iosevich, A. Mayeli, 
Gabor orthogonal bases on convexity. 
Discrete Anal. \textbf{2018}, Paper No. 19, 11 pp.

\bibitem[KL96]{KL96}
M. N. Kolountzakis, J. C. Lagarias, 
Structure of tilings of the line by a function. 
Duke Math. J. \textbf{82} (1996), no. 3, 653--678.

\bibitem[KL16]{KL16}
M. N. Kolountzakis, N. Lev, 
On non-periodic tilings of the real line by a function. 
Int. Math. Res. Not. IMRN 2016, no. 15, 4588--4601.

\bibitem[KL21]{KL21}
M. N. Kolountzakis, N. Lev, 
Tiling by translates of a function: results and open problems. 
Discrete Anal. 2021, Paper No. 12, 24 pp.

\bibitem[LM19]{LM19}
C.-K. Lai, A. Mayeli, 
Non-separable lattices, Gabor orthonormal bases and tilings.
J. Fourier Anal. Appl. \textbf{25} (2019), no. 6, 3075--3103.

\bibitem[Liu01]{Liu01}
Y. Liu, 
A characterization for windowed Fourier orthonormal basis with compact support. 
Acta Math. Sin. (Engl. Ser.) \textbf{17} (2001), no. 3, 501--506.

\bibitem[LW03]{LW03}
Y. Liu, Y. Wang, 
The uniformity of non-uniform Gabor bases. 
Adv. Comput. Math. \textbf{18} (2003), no. 2--4, 345--355.

\end{thebibliography}
\end{document}